\documentclass[12pt]{amsart}

\usepackage{epsfig}
\usepackage{amsmath}
\usepackage{amssymb}
\usepackage{amscd}
\usepackage{graphicx}
\newtheorem{thm}{Theorem}[section]
\newtheorem{lem}[thm]{Lemma}
\newtheorem{ex}[thm]{Example}

\newtheorem{pro}[thm]{Proposition}

\newtheorem{de}[thm]{Definition}
\newtheorem{rem}[thm]{Remark}

\def \N {\mathbb N}

\def \Z {\mathbb Z}
\def \R {\mathbb R}

\def\A {\mathcal A}
\def \M {\mathcal M}
\def\W {\mathcal W}
\def\U {\mathcal U}
\def\V {\mathcal V}
\def\bu {{\bf U}}
\def\I {\mathcal I}
\def\J {\mathcal J}
\def\B {\mathcal B}
\numberwithin{equation}{section}
\def\htop{h_{\rm top}}
\begin{document}
\baselineskip 13.5pt

\title[Variational principles of topological entropies]{Variational principles for  topological  entropies of subsets}
%\date{Dec 30, 2009}
\date{}

\author{De-Jun Feng and Wen Huang}
\address{Department of Mathematics, The Chinese University of Hong
Kong Shatin, Hong Kong.}

\email{djfeng@math.cuhk.edu.hk}

\address{Department of Mathematics, University of Science and Technology of
China, Hefei, Anhui, 230026, P.R. China}
\email{wenh@mail.ustc.edu.cn}



\begin{abstract} Let $(X,T)$ be a topological dynamical system. We define the measure-theoretical lower and upper entropies  $\underline{h}_\mu(T)$,  $\overline{h}_\mu(T)$ for any $\mu\in M(X)$, where $M(X)$ denotes the collection of all Borel probability measures
 on $X$. For any non-empty compact subset $K$ of $X$, we
show that
$$\htop^B(T, K)= \sup \{\underline{h}_\mu(T): \mu\in M(X),\; \mu(K)=1\}, $$
$$\htop^P(T, K)= \sup \{\overline{h}_\mu(T): \mu\in M(X),\; \mu(K)=1\}. $$
where  $\htop^B(T, K)$ denotes Bowen's topological entropy of $K$, and $\htop^P(T, K)$  the packing topological entropy of $K$. Furthermore, when $\htop(T)<\infty$,
the first equality remains valid when $K$ is replaced by an arbitrarily analytic subset of $X$. The second equality always extends to any analytic subset of $X$.
\end{abstract}

\maketitle

\section{Introduction}
\label{S-1}
Throughout this  paper, by a {\it topological dynamical system}
(TDS) $(X, T)$ we mean a compact metric space $X$ together with a
continuous self-map $T: X\rightarrow X$. Let $M(X)$
, $M(X, T)$, and
$E(X, T)$ denote respectively the sets of all Borel probability
measures, $T$-invariant Borel probability measures, and
$T$-invariant ergodic Borel probability measures on $X$. By a {\it measure theoretical dynamical system}
(m.t.d.s.) we mean $(Y, \mathcal{C},\nu,T)$, where $Y$ is a set,
$\mathcal{C}$ is a $\sigma$-algebra over $Y$, $\nu$ is a probability
measure on $\mathcal{C}$ and $T$ is a measure preserving
transformation. A probability measure $\mu \in M(X,T)$ induces a
m.t.d.s. $(X,\mathcal{B}_X,\mu,T)$ or just $(X,\mu,T)$, where
$\mathcal{B}_X$ is the $\sigma$-algebra of Borel subsets of $X$.

In 1958 Kolmogorov \cite{K} associated to any m.t.d.s. $(Y,
\mathcal{C},\nu,T)$ an isomorphic invariant, namely the
measure-theoretical entropy $h_\nu(T)$. Later on in 1965, Adler, Konheim and McAndrew
\cite{AKM65} introduced for any TDS $(X,T)$ an analogous notion of topological
entropy $\htop(T)$, as an invariant of topological conjugacy. There is a basic relation between topological entropy and measure-theoretic entropy: if $(X,T)$ is a TDS, then $\htop(T)=\sup\{h_\mu(T):\; \mu\in M(X,T)\}$.  This variational principle was proved by  Goodman
\cite{Gm}, and  plays a fundamental role in ergodic theory and
dynamical systems (cf. \cite{Pes97, Wal82}).

 In 1973, Bowen \cite{Bow73} introduced   the  topological  entropy
$\htop^B(T, Z)$ for any set  $Z$ in a TDS  $(X,T)$ in a way  resembling Hausdorff dimension,  which we call {\it Bowen's topological entropy}  (see Sect.~\ref{S-2} for the definition). In particular, $\htop^B(T, X)=\htop(T)$.   Bowen's topological entropy plays a key role in topological dynamics and dimension theory \cite{Pes97}.

%In the same article, Bowen proved that for $\mu \in E(X,T)$,
%$\htop^B(T, G_\mu)=h_\mu(T)$, where $G_\mu$ is the set of
 %generic points of $\mu$.

A question arises naturally whether there is certain
variational relation between Bowen's topological entropy and measure-theoretic entropy for arbitrary non-invariant compact set, or Borel set in general.  However, when $K\subseteq X$  is $T$-invariant but not compact, or $K$ is  compact but not $T$-invariant,  it may happen that $\htop^B(T,K)>0$ but
$\mu(K)=0$ for any $\mu\in M(X,T)$ (see Example \ref{ex-2}).
 Hence we don't expect to have
such variational principle on the class $M(X,T)$.  For our purpose,
we need to define the measure-theoretic entropy for  elements in
$M(X)$.

Fix  a compatible metric $d$ on $X$. For any $n \in \N$, the {\it $n$-th Bowen metric $d_n$ on $X$} is defined by
\begin{equation}
\label{e-2.1}
 d_n(x, y) =\max\left\{d\left(T^k(x), T^k(y)\right):\;k =0, \ldots,n-1 \right\}.
 \end{equation}
 For every $\epsilon>0$ we denote by $B_n(x, \epsilon)$, $\overline{B}_n(x, \epsilon)$ the open (resp. closed) ball
of radius $\epsilon$ in the metric $d_n$ around $x$, i.e.,
\begin{equation}
\label{e-bow}
B_n(x,
\epsilon) =\{y \in X : \ d_n(x, y) < \epsilon\},\quad \overline{B}_n(x,
\epsilon) =\{y \in X : \ d_n(x, y) \leq \epsilon\}.
\end{equation}
Following the idea of Brin and Katok \cite{BK}, we give the
following.
\begin{de}
{\rm Let $\mu\in M(X)$.  The {\it measure-theoretical lower and
upper entropies} of $\mu$ are defined respectively by
$$\underline{h}_\mu(T)=\int \underline{h}_\mu(T,x) \;d\mu(x),\quad
\overline{h}_\mu(T)=\int \overline{h}_\mu(T,x) \;d\mu(x),$$ where
\begin{equation*}
\begin{split}
&\underline{h}_\mu(T,x)=\lim\limits_{\epsilon\rightarrow
0}\liminf \limits_{n\rightarrow
+\infty}-\frac{1}{n}\log\mu(B_n(x,\epsilon)),\\
&\overline{h}_\mu (T,x)=\lim\limits_{\epsilon\rightarrow 0}\limsup
\limits_{n\rightarrow +\infty}-\frac{1}{n}\log\mu(B_n(x,\epsilon)).
\end{split}
\end{equation*}

}
\end{de}

Brin and Katok \cite{BK} proved that for any $\mu\in M(X,T)$,
$\underline{h}_\mu(T,x)=\overline{h}_\mu(T,x)$ for $\mu$-a.e
$x\in X$, and $\int \underline{h}_\mu(T,x)\; d\mu(x)=h_\mu(T)$.
Hence for  $\mu\in M(X,T)$,
$$\underline{h}_\mu(T)=\overline{h}_\mu(T)=h_\mu(T).$$

\medskip
To formulate our results, we need to introduce an additional notion.
A set in a metric space is said to be {\it analytic} if it is a continuous image of the set ${\mathcal N}$ of infinite sequences of
natural numbers (with its product topology).
It is known that in a Polish space,  the analytic subsets are closed under countable unions and intersections, and any Borel set is analytic (cf. Federer \cite[2.2.10]{Fed69}).

The main results of this paper are the following two theorems.

\begin{thm}
\label{thm-1.1} Let $(X,T)$ be a TDS.
\begin{itemize}
\item[(i)] If
$K\subseteq X$ is  non-empty and  compact, then
$$\htop^B(T, K)=\sup \{\underline{h}_\mu(T): \mu \in M(X), \; \mu(K)=1\}.$$
\item[(ii)] Assume that $\htop(T)<\infty$.  If $Z\subseteq X$ is analytic, then
\begin{equation}
\htop^B(T, Z)=\sup \{\htop^B(T, K):\; K\subseteq Z \mbox{ is compact }\}.
\end{equation}
\end{itemize}
\end{thm}
\medskip

\begin{thm}
\label{thm-4.1}
Let $(X,T)$ be a TDS.
\begin{itemize}
\item[(i)] If
$K\subseteq X$ is  non-empty and  compact, then
$$\htop^P(T, K)=\sup \{\overline{h}_\mu(T): \mu \in M(X), \; \mu(K)=1\},$$
where $\htop^P(T, K)$  denotes the packing topological entropy  of $K$ (see Sect.~\ref{S-2} for the definition).

\item[(ii)]   If $Z\subseteq X$ is analytic, then
\begin{equation}
\htop^P(T, Z)=\sup \{\htop^P(T, K):\; K\subseteq Z \mbox{ is compact }\}.
\end{equation}
\end{itemize}
\end{thm}

\medskip

The above two theorems establish the variational principles for Bowen and packing topological entropies of arbitrary Borel sets in a dual manner.
They provide as a kind of extension of the classical variational principle for topological entropy of  compact invariant sets.   In the reminder of this section,
we give two examples which motivated this paper.
\begin{ex}
\label{ex-1}
{\rm
Let $(X,T)$ denote the one-sided full shift over a finite  alphabet $\{1, 2,\ldots,\ell\}$, where $\ell$ is an integer $\geq 2$. Endow  $X$ with the metric
 $d(x, y) = e^{-n}$ for $x = (x_j)_{j=1}^\infty$ and
$y =(y_j)_{j=1}^\infty$, where $n$ is the largest integer such that
$x_j = y_j$ ($1\leq j \leq n$). It is easy to check by definition  that for any $E\subseteq X$,
$$
\htop^B(T, E)=\dim_HE,\quad \htop^P(T,E)=\dim_PE,
$$
where $\dim_HE,\dim_PE$ denote respectively the Hausdorff dimension and the packing dimension of $E$ in the ultra-metric space $(X,d)$  (cf. \cite{Mat95}).
It is a well known fact in geometric measure theory (cf. \cite{Mat95}) that, for any analytic set $Z\subseteq X$ with $\dim_HZ>0$, and any $0\leq s<\dim_HZ$, $0\leq t<\dim_PZ$, there exist compact sets $K_1, K_2\subset Z$
such that
$$
0<{\mathcal H}^s(K_1)<\infty, \quad 0<{\mathcal P}^t(K_2)<\infty,
$$
where ${\mathcal H}^s$, ${\mathcal P}^s$ denote respectively the $s$-dimensional Hausdorff measure and packing measure, and hence $\dim_HK_1=s$, $\dim_PK_2=t$. Furthermore, for ${\mathcal H}^s$-a.e $x\in K_1$, and ${\mathcal P}^t$-a.e $y\in K_2$,
$$
\liminf_{r\to 0} \frac{\log {\mathcal H}^s (K_1\cap B_r(x))}{\log r}=s,\;\quad \limsup_{r\to 0}\frac{\log {\mathcal P}^t(K_2\cap B_r(x))}{\log r}=t,
$$
where $B_r(x)$ denotes the open ball centered at $x$ of radius $r$.  This  can derive Theorems \ref{thm-1.1}-\ref{thm-4.1} in the full shift case with some additional density arguments as in  \cite[p.99, Exercises 6-7]{Mat95}.

}

\end{ex}

\begin{ex}
\label{ex-2}
{\rm
Again let $(X,T)$ denote the one-sided full shift over a finite  alphabet $\{1, 2,\ldots,\ell\}$.
Define $\varphi: X\to \R$  as
$$\varphi(x)=\left\{
\begin{array}{ll}
1& \mbox{ if }x_1=1\\
0 & \mbox{ otherwise}
\end{array}
\right.
$$
for $x=(x_i)_{i=1}^\infty\in X$. Let $E$ denote the  set of ``non-typical points'' associated with  the Birkhoff average of  $\varphi$, i.e.,
$$E=\left\{x\in X:\; \liminf_{n\to \infty}\frac{1}{n}\sum_{i=0}^{n-1}\varphi(T^ix)\neq \limsup_{n\to \infty}\frac{1}{n}\sum_{i=0}^{n-1}\varphi(T^ix)\right\}.
$$
It is easy to see that $E$ is  $T$-invariant and  Borel.
By the Birkhoff ergodic theorem, $\mu(E)=0$ for any $\mu\in M(X, T)$.  However $\htop^B(T,E)=\htop(T)=\log \ell$ (cf. \cite{BaSc00}).
 Furthermore, as we mention in Example \ref{ex-1} that for any $0\leq s<\log \ell$, there exists a compact set $K\subset E$ such that $\htop^B(T,K)=\dim_HK=s$.}
\end{ex}

In our proofs of Theorems \ref{thm-1.1}-\ref{thm-4.1}, we  use and extend some ideas and techniques in  geometric measure theory and topological dynamical systems. We remark that the assumption $\htop^B(T)<\infty$ in Theorem \ref{thm-1.1}(ii) can be weaken somewhat (see Remark \ref{rem-lind}). However it remains open
whether this assumption can be removed.

The paper is organized as follows. In Sect.~\ref{S-2} we  give the definitions and some basic properties of several topological entropies of subsets in a TDS: upper capacity topological entropy, Bowen's topological entropy, the packing topological entropy.
In Sect.~\ref{S-3}, we prove Theorem \ref{thm-1.1}. In Sect.~\ref{S-4},  we prove Theorem \ref{thm-4.1}.

\section{Topological entropies of subsets}\label{S-2}

In this section, we  give  the definitions and some basic properties of several topological entropies of subsets in a TDS: upper capacity topological entropy, Bowen's topological entropy and packing topological entropy.

Let $(X, d)$ be a compact metric space and $T : X \to X$  a
continuous transformation. Let $d_n$ and $B_n(x,\epsilon)$ be defined as in \eqref{e-2.1}-\eqref{e-bow}.
\subsection{Upper capacity topological entropy}

Let $Z\subseteq X$ be a non-empty set. For
$\epsilon>0$, a set $E\subset Z$ is called a
{\it $(n,\epsilon)$-separated set} of $Z$ if $x,y\in E, x\neq y$
implies $d_n(x,y)>\epsilon$; $E\subseteq X$ is called {\it
$(n,\epsilon)$-spanning set} of $Z$, if for any $x\in Z$, there
exists $y\in E$ with $d_n(x,y)\le \epsilon$.   Let $r_n(Z,\epsilon)$ denote the
largest cardinality of  $(n,\epsilon)$-separated sets for $Z$, and $\tilde{r}_n(Z,\epsilon)$ the smallest cardinality of $(n,\epsilon)$-separated sets of $Z$. The {\it upper capacity topological entropy of $T$ restricted on $Z$}, or simply, the
{\it upper capacity topological entropy of $Z$} is defined as
$$\htop^{UC}(T, Z)=\lim_{\epsilon \to 0} \limsup\limits_{n\rightarrow \infty}
\frac{1}{n} \log r_n (Z, \epsilon)=\lim_{\epsilon \to 0} \limsup\limits_{n\rightarrow \infty}
\frac{1}{n} \log \tilde{r}_n (Z, \epsilon).
$$
We remark that the second equality holds for each $Z\subseteq X$ (cf. \cite[P. 169]{Wal82}).
The quantity $\htop^{UC}(T, Z)$ is the straightforward generalization of the Adler-Konheim-McAndrew definition \cite{AKM65} of the topological entropy to arbitrary subsets.

\subsection{Bowen's topological entropy}
\label{S-2.2}
Suppose that $\U$ is a finite open cover of
$X$. Denote $\mbox{diam}(\U):=\max\{\mbox{diam}(U): U\in \U\}$. For $n\geq 1$ we denote by $\W_n(\U)$ the collection of
strings $\bu=U_1\ldots U_n$ with $U_i\in \U$. For $\bu\in
\W_n(\U)$ we call the integer $m(\bu)=n$ the {\it length of $\bu$} and
define
\begin{eqnarray*}
X(\bu) & =&U_{1}\cap T^{-1}U_2\cap \ldots \cap
T^{-(n-1)}U_n\\
 &=&\left\{x\in X:\; T^{j-1}x\in U_j \mbox{
for } j=1,\ldots,n\right\}.
\end{eqnarray*}
Let $Z\subseteq X$. We say that $\Lambda\subset\bigcup_{n\geq 1}\W_n(\U)$ {\it covers $Z$} if
$\bigcup_{\bu\in \Lambda}X(\bu)\supset Z$.  For $s\in \R$, define
$$
\M^s_N(\U, Z)=\inf_{\Lambda} \sum_{\bu\in \Lambda} e^{-s
m(\bu)},
$$
where the infimum is taken over all $\Lambda \subset\bigcup_{j\geq
N}\W_j(\U)$ that cover $Z$.
Clearly $\M^s_N(\U, \cdot)$ is a finite outer measure on $X$, and
\begin{equation}
\label{e-open}
\M^s_N(\U, Z)=\inf\{ \M^s_N(\U, G):\; G\supset Z,\; G \mbox{ is open}\}.
\end{equation}
  Note that $\M^s_N(\U, Z)$ increases as $N$ increases.
Define $\M^s(\U, Z)=\lim_{N\to \infty}\M^s_N(\U, Z)$ and
$$
\htop^B(T,\U, Z)=\inf\{s:\; \M^s(\U, Z)=0\}=\sup\{s:\; \M^s(
\U, Z)=+\infty\}. $$
Set
\begin{equation}
\label{e-e1}\htop^B(T, Z)=\sup_{\U}\htop^B(T,\U, Z),
\end{equation}
where $\U$ runs over finite open covers of $Z$. We  call $\htop^B(T, Z)$ the {\it Bowen's topological entropy of $T$ restricted to $Z$}
or, simply, the {\it topological entropy of $Z$}.  This quantity was first introduced by Bowen in \cite{Bow73}. It is known (see, i.e.  \cite[Theorem 11.1]{Pes97}) that
 \begin{equation}
 \label{e-bowen}
 \sup_{\U}\htop^B(T,\U, Z)
=\lim_{{\rm diam}(\U)\to 0}\htop^B(T,\U, Z).
\end{equation}

Bowen's topological entropy of subsets can  be defined in an alternative way. For $Z \subseteq X$,   $s\geq 0$, $N\in \N$  and  $\epsilon> 0$, define
\begin{equation*}\label{e-top}
\M^s_{N, \epsilon}(Z) =\inf\sum_i
\exp(-sn_i),
\end{equation*}
 where the infimum is taken over all finite or countable families $\{B_{n_i}(x_i,\epsilon)\}$ such that $x_i\in X$, $n_i\geq N$ and $\bigcup_iB_{n_i}(x_i,\epsilon)\supseteq Z$. The quantity $\M^s_{N, \epsilon}(Z)$ does not decrease as $N$ increases and $\epsilon$ decreases, hence the following
limits exist:
$$\M^s_\epsilon(Z) = \lim_{N\to \infty} \M^s_{N,\epsilon}(Z),\quad \M^s(Z)=\lim_{\epsilon\to 0}\M^s_\epsilon(Z).$$
Bowen's topological entropy $\htop^B(T,Z)$ can be equivalently defined as    a critical value
of the parameter s,
where $\M^s(Z)$ jumps from $\infty$ to $0$, i.e.
\[
\M^s(Z) = \left\{
\begin{array}{ll}
0, & s > \htop^B(T,Z),\\
\\

\infty,& s < \htop^B(T,Z).
\end{array}
\right.
\]
 For details, see \cite[Page 74]{Pes97}.

\subsection{Packing topological entropy}
Let $Z\subseteq X$. For $s\geq 0$, $N\in \N$ and $\epsilon>0$, define
$$
P^s_{N,\epsilon}(Z)=\sup\sum_i
\exp(-sn_i),
$$
 where the supermum is taken over all finite or countable pairwise disjoint families $\{\overline{B}_{n_i}(x_i,\epsilon)\}$ such that $x_i\in Z$, $n_i\geq N$ for all $i$, where
 $$\overline{B}_n(x,\epsilon):=\{y\in X:\; d_n(x,y)\leq \epsilon\}.
 $$
 The quantity $P^s_{N, \epsilon}(Z)$ does not decrease as $N,\epsilon$ decrease, hence the following
limits exist:
$$P^s_\epsilon(Z) = \lim_{N\to \infty} P^s_{N,\epsilon}(Z).$$
Define
$$
{\mathcal P}^s_\epsilon(Z)=\inf\left\{\sum_{i=1}^\infty P^s_\epsilon(Z_i):\; \bigcup_{i=1}^\infty Z_i\supseteq Z\right\}.
$$
Clearly, ${\mathcal P}^s_\epsilon$ satisfies the following property:   if $Z\subseteq \bigcup_{i=1}^\infty Z_i$, then ${\mathcal P}^s_\epsilon(Z)\leq \sum_{i=1}^\infty {\mathcal P}^s_\epsilon(Z_i)$.
There exists a critical value
of the parameter s, which we will denote by $\htop^P(T,Z,\epsilon)$,
where ${\mathcal P}^s_\epsilon(Z)$ jumps from $\infty$ to $0$, i.e.
\[
{\mathcal P}^s_\epsilon(Z) = \left\{
\begin{array}{ll}
0, & s > \htop^P(T,Z,\epsilon),\\
\\
\infty,& s < \htop^P(T,Z,\epsilon).
\end{array}
\right.
\]
Note that  $\htop^P(T,Z,\epsilon)$ increases when $\epsilon$ decreases.
 We  call
$$\htop^P(T, Z):=\lim_{\epsilon\to 0}\htop^P(T,Z,\epsilon)$$ the {\it packing topological entropy of $T$
restricted to $Z$} or, simply, the {\it packing topological entropy
of $Z$}, when there is no confusion about $T$. This quantity is defined in way which resembles the packing dimension. We remark that an equivalent definition
of packing topological entropy was given earlier in \cite{HYZ1}.

\subsection{Some basic  properties}
\begin{pro}
\begin{itemize}
\item[(i)]
For $Z\subseteq Z'$,
$$\htop^{UC}(T, Z)\leq \htop^{UC}(T, Z'),\; \htop^{B}(T, Z)\leq \htop^{B}(T, Z'),\; \htop^{P}(T, Z)\leq \htop^{P}(T, Z').$$

\item[(ii)]
For $Z\subseteq \bigcup_{i=1}^\infty Z_i$, $s\geq 0$ and $\epsilon>0$, we have

$$
\M^s_\epsilon(Z)\leq \sum_{i=1}^\infty \M^s_\epsilon(Z_i),\; \htop^{B}(T, Z)\leq \sup_{i\geq 1}\htop^{B}(T, Z_i),\;
 \htop^{P}(T, Z)\leq \sup_{i\geq 1}\htop^{P}(T, Z_i).
$$

\item[(iii)] For any $Z\subseteq X$,
$
\htop^B(T, Z)\leq \htop^P(T,Z)\leq \htop^{UC}(T,Z).
$

\item[(iv)] Furthermore, if $Z$ is $T$-invariant and compact, then
 $$
\htop^B(T, Z)= \htop^P(T,Z)= \htop^{UC}(T,Z).
$$

\end{itemize}
\end{pro}
\begin{proof}
 (i) and (ii) follow directly from the definitions of topological entropies. To see (iii), let $Z\subseteq X$ and assume $0<s<\htop^B(T,Z)$. For any $n\in \N$ and $\epsilon>0$, let $R=R_n(Z,\epsilon)$ be the largest number so that there is a disjoint family $\{\overline{B}_n(x_i,\epsilon)\}_{i=1}^R$ with $x_i\in Z$. Then it is easy to see that for any  $\delta>0$,
$$\bigcup_{i=1}^R\overline{B}_n(x_i,2\epsilon+\delta)\supseteq Z,
$$
which implies that $\M^s_{n,2\epsilon+\delta}(Z)\leq  Re^{-ns}\leq P^s_{n,\epsilon}(Z)$ for any $s\geq 0$, and hence $\M^s_{2\epsilon+\delta}(Z)\leq P^s_\epsilon(Z)$.
By (ii), $\M^s_{2\epsilon+\delta}(Z)\leq {\mathcal P}^s_\epsilon(Z)$. Since $0<s<\htop^B(T,Z)$, we have $\M^s(Z)=\infty$ and thus $\M^s_{2\epsilon+\delta}(Z)\geq 1$ when
$\epsilon$ and $\delta$ are small enough. Hence  ${\mathcal P}^s_\epsilon(Z)\geq 1$ and $\htop^P(T,Z,\epsilon)\geq s$ when $\epsilon$ is small. Therefore $\htop^P(T,Z)=\lim_{\epsilon\to 0} \htop^P(T,Z,\epsilon)\geq s$. This implies that $\htop^B(T, Z)\leq \htop^P(T,Z)$.

Next we show that $\htop^{P}(T, Z)\leq \htop^{UC}(T, Z)$. Our argument is modified slightly from the proof of \cite[Lemma 3.7]{Fal03}.
Assume that  $\htop^{P}(T, Z)>0$; otherwise there is nothing left to prove. Choose $0<t<s<\htop^{P}(T, Z)$. Then there exists $\delta>0$ such that for $0<\epsilon<\delta$,  $\htop^{P}(T, Z, \epsilon)>s$ and thus $P^s_\epsilon(Z)\geq {\mathcal P}_\epsilon^s(Z)=\infty$. Thus for any $N$, there exists a countable pairwise disjoint families $\{\overline{B}_{n_i}(x_i,\epsilon)\}$ such that $x_i\in Z$, $n_i\geq N$ for all $i$, and $1< \sum_ie^{-n_is}$. For each $k$, let $m_k$ be the number of $i$ so that $n_i=k$. Then we have
$$
1<\sum_{k=N}^\infty m_k e^{-ks}.
$$
There must be some $k\geq N$ with $m_k>e^{kt}(1-e^{t-s})$, otherwise the above sum is at most $\sum_{k=1}^\infty e^{kt-ks}(1-e^{t-s})<1$.   Let $r_k(Z,\epsilon)$ denote the
largest cardinality of  $(k,\epsilon)$-separated sets for $Z$. Then $r_k(Z,\epsilon)\geq m_k>e^{kt}(1-e^{t-s})$. Hence $\limsup\limits_{n\rightarrow \infty}
\frac{1}{n} \log r_n (Z, \epsilon)\geq t$. Letting $\epsilon\to 0$, we obtain $\htop^{UC}(T, Z)\geq t$. This is true for any $0<t<\htop^P(T,Z)$ so $\htop^{UC}(T, Z)\geq
\htop^P(T,Z)$.

When $Z\subseteq X$ is $T$-invariant and compact, Bowen  \cite{Bow73} proved that  $\htop^B(T, Z)= \htop^{UC}(T,Z)$; this together with  (iii) yields  (iv).
\end{proof}

\section{Variational principle for Bowen's topological entropy of subsets}
\label{S-3}

\subsection{Weighted topological entropy}

For any function $f:X\to [0,\infty)$,
 $N\in \N$ and $\epsilon>0$, define
\begin{equation}
\label{e-wei}
\W^s_{N,\epsilon}(f)=\inf\sum_ic_i\exp(-sn_i),
\end{equation}
 where the infimum is taken over all finite or countable families  $\{(B_{n_i}(x_i,\epsilon), c_i)\}$ such that
 $0<c_i<\infty$, $x_i\in X$,  $n_i\geq N$ and
 $$
  \sum_{i}c_i\chi_{B_i}\geq f,
 $$
 where $B_i:=B_{n_i}(x_i,\epsilon)$, and  $\chi_A$ denotes the characteristic function of $A$, i.e, $\chi_A(x)=1$ if $x\in A$ and $0$ if $x\in X\backslash  A$.

For $Z\subseteq X$ and $f=\chi_Z$ we set $\W^s_{N,\epsilon}(Z)=\W^s_{N,\epsilon}(\chi_Z)$.
The quantity $\W^s_{N, \epsilon}(Z)$ does not decrease as $N$ increases and $\epsilon$ decreases, hence the following
limits exist:
$$\W^s_\epsilon(Z) = \lim_{N\to \infty} \W^s_{N,\epsilon}(Z),\quad \W^s(Z)=\lim_{\epsilon\to 0}\W^s_\epsilon(Z).$$
We remark that  $\W^s$ is defined in a way which resembles the weighted Hausdorff measure in geometric measure theory (cf. \cite{Fed69, Mat95}).
Clearly, there exists a critical value
of the parameter s, which we will denote by $\htop^{WB}(T,Z)$,
where $\W^s(Z)$ jumps from $\infty$ to $0$, i.e.
\[
\W^s(Z) = \left\{
\begin{array}{ll}
0, & s > \htop^{WB}(T,Z),\\
\infty,& s < \htop^{WB}(T,Z).
\end{array}
\right.
\]
 We  call
$\htop^{WB}(T, Z)$ the {\it weighted Bowen's topological entropy of $T$ restricted to
$Z$} or, simply, the {\it weighted Bowen's topological entropy of $Z$}.

\subsection{Equivalence of $\htop^B$ and $\htop^{WB}$}

The following properties about  $\M^s$ (cf. Sect. \ref{S-2.2}) and $\W^s$ can be verifies directly from the definitions.
\begin{pro}
\label{pro-1.1}
\begin{itemize}
\item[(i)] For any $s\geq 0$, $N\in \N$ and $\epsilon>0$, both $\M^s_{N,\epsilon}$ and $\W^s_{N,\epsilon}$ are outer measures on $X$.
\item[(ii)] For any $s\geq 0$, both $\M^s$ and $\W^s$ are metric outer measures on $X$.
\end{itemize}
\end{pro}

We remark that $\M^s$ and $\W^s$ depend not only $s$ but also the TDS $(X,T)$. However, $\M^s$ and $\W^s$ are purely topological and  independent of the special choice of the metric $d$.

The main result of this subsection is the following.

\begin{pro}
\label{pro-2.1}
Let  $Z\subseteq X$. Then for any $s\geq 0$ and  $\epsilon,\delta>0$, we have
\begin{equation*}
\label{e-ine}
\M^{s+\delta}_{N,6\epsilon}(Z)\leq \W^s_{N,\epsilon}(Z)\leq \M^{s}_{N,\epsilon}(Z),
\end{equation*}
when $N$ is large enough. As a result, $\M^{s+\delta}(Z)\leq \W^s(Z)\leq \M^s(Z)$ and $\htop^B(T,Z)=\htop^{WB}(T,Z)$.
\end{pro}

To prove Proposition \ref{pro-2.1}, we need the following lemma.

\begin{lem}
[\cite{Mat95}, Theorem 2.1]
\label{lem-2.1}

Let $(X, d)$ be a compact metric space and ${\mathcal
B}=\{B(x_i,r_i)\}_{i\in \mathcal I}$ be a family of closed (or open)
balls in $X$. Then there exists a finite or countable subfamily
${\mathcal B'}=\{B(x_i,r_i)\}_{i\in {\mathcal I}'}$ of pairwise
disjoint balls in ${\mathcal B}$ such that
$$\bigcup_{B\in {\mathcal B}} B\subseteq \bigcup_{i\in {\mathcal I}'}B(x_i,5r_i).$$
\end{lem}

\bigskip

\begin{proof}[Proof of Proposition \ref{pro-2.1}] Let $Z\subseteq X$, $s\geq 0$, $\epsilon, \delta>0$. Taking  $f=\chi_Z$ and $c_i\equiv 1$ in the definition \eqref{e-wei}, we see that
 $\W^s_{N,\epsilon}(Z)\leq \M^{s}_{N,\epsilon}(Z)$ for each $N\in \N$. In the following, we prove that $\M^{s+\delta}_{N,6\epsilon}(Z)\leq \W^s_{N,\epsilon}(Z)$
 when $N$ is large enough.

 Assume that $N\geq 2$ such that $n^2e^{-n\delta}\le 1$ for $n\geq N$.
 Let $\{(B_{n_i}(x_i,\epsilon), c_i)\}_{i\in \mathcal I}$ be a  family so that $\I\subseteq \N$, $x_i\in X$,  $0<c_i<\infty$,
  $n_i\geq N$ and
 \begin{equation} \label{e-gez}
  \sum_{i}c_i\chi_{B_i}\geq \chi_Z,
 \end{equation}
 where $B_i:=B_{n_i}(x_i,\epsilon)$. We show below that
 \begin{equation}
 \label{e-key}
 \M^{s+\delta}_{N,6\epsilon}(Z)\leq \sum_{i\in \I}c_ie^{-n_is},
 \end{equation}
 which implies  $\M^{s+\delta}_{N,6\epsilon}(Z)\leq \W^s_{N,\epsilon}(Z)$.

 Denote  $\I_n:=\{i\in \I:\; n_i=n\}$ and   $\I_{n,k}=\{i\in \I_n:\; i\leq k\}$ for $n\geq N$ and $k\in \N$.
 Write for brevity $B_i:=B_{n_i}(x_i, \epsilon)$ and  $5B_i:=B_{n_i}(x_i, 5\epsilon)$ for $i\in \I$. Obviously we may assume $B_i\neq B_j$ for $i\neq j$.
 For $t>0$, set
 \begin{equation*}
 \begin{split}
  Z_{n,t}&=\Big\{x\in Z:\; \sum_{i\in \I_{n}}c_i\chi_{B_i}(x)>t\Big\} \quad\mbox{ and }\\
 Z_{n,k,t}&=\Big\{x\in Z:\; \sum_{i\in \I_{n,k}}c_i\chi_{B_i}(x)>t\Big\}.
  \end{split}
 \end{equation*}
 We divide the proof of \eqref{e-key} into the following three steps.

  {\sl Step 1. For each $n\geq N$,  $k\in \N$ and $t>0$, there exists a finite set $\J_{n,k, t}\subseteq \I_{n,k}$ such that
 the balls $B_i$ {\rm ($i\in \J_{n,k, t}$)} are pairwise disjoint, $Z_{n,k,t}\subseteq \bigcup_{i\in \J_{n,k, t}}5B_i$ and
 $$\#(\J_{n,k,t})e^{-ns}\leq \frac{1}{t}\sum_{i\in \I_{n,k}}c_ie^{-ns}.$$
 }
 To prove the above result, we adopt the method of Federer \cite[2.10.24]{Fed69} used in the study of weighted Hausdorff measures (see also Mattila \cite[Lemma 8.16]{Mat95}). Since  $\I_{n,k}$ is finite, by approximating the $c_i$'s from above,
 we may assume that each $c_i$ is a positive rational, and then multiplying with a common denominator we may assume that each $c_i$ is
 a positive integer. Let $m$ be the least integer with $m\geq t$. Denote $\B=\{B_i,\; i\in \I_{n,k}\}$ and define $u: \B\to \Z$ by $u(B_i)=c_i$.
 We define by induction integer-valued functions $v_0,v_1,\ldots, v_m$ on $\B$ and sub-families $\B_1,\ldots, \B_m$ of $\B$ starting with $v_0=u$.
 Using Lemma \ref{lem-2.1} (in which we take the metric $d_n$ instead of $d$) we  find a pairwise disjoint subfamily $\B_1$ of $\B$ such that
 $\bigcup_{B\in \B}B\subseteq \bigcup_{B\in \B_1} 5B$, and hence $Z_{n,k,t}\subseteq \bigcup_{B\in \B_1} 5B$. Then by repeatedly using Lemma \ref{lem-2.1}, we can  define inductively for $j=1, \ldots, m$,  disjoint subfamilies $\B_j$ of $\B$ such that
 $$\B_j\subseteq \{B\in \B:\; v_{j-1}(B)\geq 1\},\quad Z_{n,k,t}\subseteq \bigcup_{B\in \B_j} 5B$$
 and the functions $v_j$ such that
 $$
 v_j(B)=\left\{\begin{array}{ll}
 v_{j-1}(B)-1 & \mbox { for } B\in \B_j,\\
  v_{j-1}(B) & \mbox { for } B\in \B\backslash \B_j.
\end{array}
\right.
$$
This is possible since for $j<m$,
$Z_{n,k,t}\subseteq \big\{x: \sum_{B\in \B:\; B\ni x}v_j(B)\geq m-j\big\}$,
whence every $x\in Z_{n,k,t}$ belongs to some ball $B\in \B$ with $v_j(B)\geq 1$. Thus
\begin{eqnarray*}
 \sum_{j=1}^m \#(\B_j) e^{-ns}&=&\sum_{j=1}^m \sum_{B\in \B_j} (v_{j-1}(B)-v_j(B)) e^{-ns} \\
&\leq & \sum_{B\in \B}\sum_{j=1}^m (v_{j-1}(B)-v_j(B)) e^{-ns}\leq \sum_{B\in \B} u(B)e^{-ns}=\sum_{i\in \I_{n,k}} c_ie^{-ns}.
\end{eqnarray*}
Choose $j_0\in \{1,\ldots, m\}$ so that $\#(\B_{j_0})$ is the smallest. Then $$\#(\B_{j_0}) e^{-ns}\leq \frac{1}{m}\sum_{i\in \I_{n,k}} c_ie^{-ns}
\leq \frac{1}{t}\sum_{i\in \I_{n,k}} c_ie^{-ns}.$$ Hence  $\J_{n,k,t}=\{i\in \I:\; B_i\in \B_{j_0}\}$ is desired.

{\sl Step 2. For each $n\geq N$ and $t>0$, we have \begin{equation}
\label{e-2.3}
\M^{s+\delta}_{N,6\epsilon}(Z_{n,t})\leq  \frac{1}{n^2t}\sum_{i\in \I_{n}}c_ie^{-ns}.
 \end{equation}
 }
To see this, assume $Z_{n,t}\neq \emptyset$; otherwise this is nothing to prove.  Since $Z_{n,k,t}\uparrow Z_{n,t}$, $Z_{n, k,t}\neq \emptyset$ when $k$ is large enough.
 Let $\J_{n,k,t}$ be the sets constructed in Step 1. Then $\J_{n,k,t}\neq \emptyset$ when $k$ is large enough.
 Define $E_{n,k,t}=\{x_i:\; i\in \J_{n,k,t}\}$.
 Note that the family of all non-empty compact subsets of $X$ is compact with respect to the Hausdorff distance (cf. Federer \cite[2.10.21]{Fed69}). It follows that there is a subsequence $(k_j)$ of natural numbers and a non-empty compact set $E_{n,t}\subset X$ such that $E_{n,k_j,t}$ converges to $E_{n,t}$ in  the Hausdorff distance as $j\to \infty$. Since any two points in $E_{n,k,t}$ have a distance (with respect to $d_n$) not less than $\epsilon$, so do the points in $E_{n,t}$. Thus $E_{n,t}$ is a finite set, moreover, $\#(E_{n,k_j,t})=\#(E_{n,t})$ when $j$ is large enough.  Hence
 $$
 \bigcup_{x\in E_{n,t}} B_n(x,5.5\epsilon)\supseteq  \bigcup_{x\in E_{n,k_j,t}} B_n(x,5\epsilon)=\bigcup_{i \in \J_{n,k_j,t}} 5 B_i\supseteq Z_{n, k_j, t}
 $$
 when $j$ is large enough, and thus $\bigcup_{x\in E_{n,t}} B_n(x,6\epsilon)\supseteq Z_{n,t}$. By the way, since
 $\#(E_{n,k_j,t})=\#(E_{n,t})$ when $j$ is large enough, we have $\#(E_{n,t})e^{-ns}\leq \frac{1}{t}\sum_{i\in \I_{n}}c_ie^{-ns}$.
 This forces $$\M^{s+\delta}_{N,6\epsilon}(Z_{n,t})\leq \#(E_{n,t})e^{-n(s+\delta)}\leq \frac{1}{e^{n\delta}t}\sum_{i\in \I_{n}}c_ie^{-ns}\leq  \frac{1}{n^2t}\sum_{i\in \I_{n}}c_ie^{-ns}.$$

 {\sl Step 3. For any $t\in (0,1)$, we have $\M^{s+\delta}_{N,6\epsilon}(Z)\leq \frac{1}{t} \sum_{i\in \I}c_ie^{-n_is}.$ As a result,  \eqref{e-key} holds.\\
  }
To see this, fix $t\in (0,1)$. Note that $\sum_{n=N}^\infty
n^{-2}<1$. It follows that $Z\subseteq \bigcup_{n=N}^\infty Z_{n,
n^{-2}t}$ from \eqref{e-gez}. Hence by Proposition \ref{pro-1.1}(i)
and \eqref{e-2.3}, we have
$$
\M^{s+\delta}_{N,6\epsilon}(Z)\leq
\sum_{n=N}^\infty\M^{s+\delta}_{N,6\epsilon}(Z_{n, n^{-2}t})\leq
\sum_{n=N}^\infty \frac{1}{t}\sum_{i\in
\I_{n}}c_ie^{-ns}=\frac{1}{t}\sum_{i\in \I}c_ie^{-n_is},$$ which
finishes the proof of the proposition.
\end{proof}

\subsection{A dynamical Frostman's lemma and the proof of Theorem \ref{thm-1.1} (i)}
\label{S-3}
To  prove Theorem \ref{thm-1.1}(i), we need the following dynamical Frostman's lemma, which is an analogue of the classical Frostman's lemma in in compact metric space.  Our proof is adapt from Howroyd's elegant argument (cf. \cite[Theorem 2]{How95}, \cite[Theorem 8.17]{Mat95}).

\begin{lem} \label{pro-3.1} Let $K$ be a non-empty compact subset of $X$. Let $s\geq 0$, $N\in \N$ and $\epsilon>0$. Suppose that
 $c:=\W^s_{N,\epsilon}(K)>0$. Then there is a Borel probability measure $\mu$ on $X$ such that  $\mu(K)=1$ and
$$
\mu(B_n(x,\epsilon))\leq \frac{1}{c}e^{-ns},\quad \forall\; x\in X,\; n\geq N.
$$
\end{lem}
\begin{proof}  Clearly $c<\infty$. We define a function $p$ on the space $C(X)$ of continuous real-valued functions on $X$ by
\begin{equation*}
\label{zx-eq}
p(f)=(1/c)\W^s_{N,\epsilon}(\chi_K\cdot f),
\end{equation*}
where $\W^s_{N,\epsilon}$ is defined as in \eqref{e-wei}.

Let ${\bf 1}\in C(X)$ denote the constant function ${\bf 1}(x)\equiv 1$. It is easy to verify that

\begin{enumerate}

\item $p(f+g)\le p(f)+p(g)$ for any $f,g\in C(X)$.

\item $p(tf)=tp(f)$ for any $t\ge 0$ and $f\in C(X)$.

\item $p({\bf 1})=1$, $0\leq p(f)\leq \|f\|_\infty$ for any $f\in C(X)$,  and
$p(g)=0$ for $g\in C(X)$ with $g\le 0$.
\end{enumerate}
By
the Hahn-Banach theorem, we can extend  the linear functional $t\mapsto t p({\bf 1})$, $t\in \R$, from the subspace of the constant functions to a linear functional $L:\;
C(X)\to \R$ satisfying
$$L({\bf 1})=p({\bf 1})=1 \text{ and }-p(-f)\le L(f)\le p(f)    \text{ for any }f\in C(X).$$
If $f\in C(X)$ with $f\ge 0$, then $p(-f)=0$ and so $L(f)\ge 0$.
Hence combining the fact $L({\bf 1})=1$, we can use the Riesz
representation theorem to find a Borel probability measure $\mu$ on
$X$ such that $L(f)=\int f d\mu$ for $f\in C(X)$.

Now we show that $\mu(K)=1$. To see this,  for any compact set
$E\subseteq X\backslash K$, by the Uryson lemma there is $f\in C(X)$ such that
$0\leq f\leq 1$, $f(x)=1$ for $x\in E$ and $f(x)=0$ for $x\in K$.
Then $f\cdot \chi_K\equiv 0$ and thus  $p(f)=0$. Hence $\mu(E)\le L(f)\le p(f)=0$.
This shows $\mu(X\backslash K)=0$, i.e. $\mu(K)=1$.

In the end, we show that $\mu(B_n(x,\epsilon))\leq (1/c) e^{-ns}$ for any $x\in X$ and $n\geq N$. To see this, for any compact set $E\subset B_n(x,\epsilon)$,   by the Uryson lemma,  there exists $f\in C(X)$ such that $0
\le f\le 1$, $f(y)=1$ for $y\in E$ and $f(y)=0$ for $y\in X\backslash B_n(x,\epsilon)$.
Then $\mu(E)\le L(f)\le p(f)$. Since $f\cdot \chi_K\leq \chi_{B_n(x,\epsilon)}$ and $n\geq N$, we have $\W^s_{N,\epsilon}(\chi_K\cdot f)\leq e^{-ns}$ and thus
$p(f)\le \frac{1}{c} e^{-sn}$. Therefore $\mu(E)\le \frac{1}{c}e^{-ns}$. It follows that
$$
\mu(B_n(x,\epsilon))=\sup\{ \mu(E):\; E \mbox{ is a compact subset of } B_n(x,\epsilon)\}\leq  \frac{1}{c} e^{-sn}.$$
\end{proof}

\begin{rem}
\label{rem-MaWe}
{\rm There is a related known result (see, e.g. \cite{MaWe08, Shu06}) that,   for any Borel set $E\subset X$ and any Borel probability measure $\mu$ on $E$,  if $\underline{h}_\mu(T,x)\leq s$ for all $x\in E$, then $\htop^B(T,E)\leq s$; conversely if $\underline{h}_\mu(T,x)\geq s$ for all $x\in E$, then $\htop^B(T,E)\geq s$,
where
$
\underline{h}_\mu(T,x)
$ is defined as in Sect.~\ref{S-1}.}
\end{rem}

Now we are ready to prove Theorem \ref{thm-1.1}(i).
\begin{proof}[Proof of Theorem \ref{thm-1.1}(i)]  We first  show that  $\htop^B(T, K)\ge {\underline h}_\mu(T)$ for any $\mu\in M(X)$ with $\mu(K)=1$.
Let $\mu$ be a given such measure.  Write
$$
{\underline h}_\mu(T,x,\epsilon)=\liminf_{n\to \infty}-\frac{1}{n}\log \mu(B_n(x,\epsilon))
$$
for $x\in X, n\in \N$ and $\epsilon>0$.
Clearly $\underline{h}_\mu(T,x,\epsilon)$ is nonnegative and  increases as $\epsilon$ decreases. Hence by the monotone convergence theorem,
 $$\lim_{\epsilon\rightarrow 0} \int \underline{h}_\mu(T,x,\epsilon)
d\mu= \int \underline{h}_\mu(T,x)
d\mu=h_\mu(T).$$ Thus to show  $\htop^B(T, K)\ge \underline{h}_\mu(T)$, it is sufficient to show
 $\htop^B(T, K)\ge \int
\underline{h}_\mu(T,x,\epsilon) d\mu$ for each $\epsilon>0$.

Fix  $\epsilon>0$ and $\ell\in \mathbb{N}$.  Denote  $u_\ell=\min
\{\ell,\int \underline{h}_\mu(T,x,\epsilon)
d\mu(x)-\frac{1}{\ell}\}$. Then there exist a Borel set
$A_\ell\subset X$ with $\mu(A_\ell)>0$ and $N\in \mathbb{N}$ such
that
\begin{equation}
\label{e-ball}
 \mu(B_n(x,\epsilon))\le e^{-nu_\ell}, \quad \forall \; x\in A_\ell,\; n\ge N.
 \end{equation}
Now let  $\{B_{n_i}(x_i,\epsilon/2)\}$ be a countable or finite family so that $x_i\in X$, $n_i\geq N$ and $\bigcup_{i}B_{n_i}(x_i,\epsilon/2)\supset K\cap A_\ell$.
We may assume that for each $i$, $B_{n_i}(x_i,\epsilon/2)\cap (K\cap A_\ell)\neq \emptyset$, and choose $y_i\in B_{n_i}(x_i,\epsilon/2)\cap (K\cap A_\ell)$.
Then by \eqref{e-ball},
\begin{eqnarray*}
\sum_{i}e^{-n_i u_\ell}&\geq& \sum_{i} \mu(B_{n_i}(y_i,\epsilon))\geq \sum_{i} \mu(B_{n_i}(x_i,\epsilon/2))\\
&\geq& \mu(K\cap A_\ell)=\mu(A_\ell)>0.
\end{eqnarray*}
It follows that $\M^{u_\ell}(K)\geq \M^{u_\ell}_{N,\epsilon/2}(K)\geq \M^{u_\ell}_{N,\epsilon/2}(K\cap A_\ell)\geq \mu(A_\ell)$. Therefore
$\htop^B(T, K)\geq u_\ell$. Letting $\ell\to \infty$, we have
 the desired inequality  $\htop^B(T, K)\geq \int\underline{h}_\mu(T,x,\epsilon)
d\mu$. Hence $\htop^B(T, K)\geq \underline{h}_\mu(T)$.

 We next show that $\htop^B(T, K)\le \sup\{\underline{h}_\mu(T): \mu\in
M(X), \mu(K)=1\}$.  We can assume that $\htop^B(T, K)>0$, otherwise we have nothing to prove. By Proposition \ref{pro-2.1},
$\htop^{BW}(T, K)=\htop^B(T, K)$. Let $0<s<\htop^B(T, K)$. Then there exist $\epsilon>0$ and $N\in \N$ such that $c:=\W^s_{N,\epsilon}(K)>0$.
By Proposition \ref{pro-3.1},  there exists $\mu\in M(X)$ with $\mu (K)=1$ such that
$\mu(B_n(x,\epsilon))\le \frac{1}{c} e^{-sn}$ for any
$x\in X$ and $n\geq N$. Clearly $\underline{h}_\mu(T,x)\geq \underline{h}_\mu(T,x,\epsilon)\geq s$ for each $x\in X$ and hence $\underline{h}_\mu(T)\geq  \int \underline{h}_\mu(T,x) d\mu(x)\geq s$.
This finishes the proof of Theorem \ref{thm-1.1}(i).
\end{proof}

\subsection{The proof of Theorem \ref{thm-1.1}(ii)}

To prove Theorem \ref{thm-1.1}(ii), we first prove  the following.

\begin{thm}
\label{thm-6.1}
Let $(X,T)$ be a TDS. Assume that $X$ is zero-dimensional, i.e., for any $\delta>0$, $X$ has a closed-open partition with diameter less than $\delta$.
Then for any analytic set $Z\subset X$,
$$\htop(T,Z)=\sup\{\htop(T,K):\; K\subset Z,\; K \mbox{ is compact}\}.$$
\end{thm}

The following proposition is needed for the proof of Theorem \ref{thm-6.1}.

\begin{pro}
\label{pro-5.2}
Assume $\U$ is a closed-open partition of $X$. Let $N\in \N$. Then
\begin{itemize}
\item[(i)] If $E_i\uparrow E$, i.e., $E_{i+1}\supseteq E_i$ and $\bigcup_iE_i=E$,  then $$\M^s_N(\U, E)=\lim_{i\to \infty}\M^s_N(\U, E_i).$$

\item[(ii)] Assume $Z\subset X$ is analytic. Then
$$
\M^s_N(\U, Z)=\sup\{\M^s_N(\U, K):\; K\subset Z,\; K \mbox{ is compact}\}.
$$
\end{itemize}
\end{pro}
\begin{proof}
We first show that (i) implies (ii). Assume that (i) holds. Let $Z$ be analytic, i.e.,
  there exists a continuous surjective map $\phi:\;{\mathcal N}\to Z$.    Let $\Gamma_{n_1,n_2,\ldots,n_p}$ be the set of $(m_1,m_2,\ldots)\in {\mathcal N}$
such that $m_1\leq n_1$, $m_2\leq n_2$, $\ldots$, $m_p\leq n_p$ and let $Z_{n_1,\ldots, n_p}$ be the image of $\Gamma_{n_1,\ldots,n_p}$ under $\phi$.
Let $(\epsilon_p)$ be a sequence of positive numbers. Due to (i), we can pick a sequence $(n_p)$ of positive integers recursively so that $\M^s_N(\U, Z_{n_1})\geq \M^s_N(\U, Z)-\epsilon_1$ and
$$
\M^s_N(\U, Z_{n_1, \ldots, n_p})\geq \M^s_N(\U, Z_{n_1,\ldots, n_{p-1}})-\epsilon_p,\quad p=2,3,\ldots
$$
Hence $\M^s_N(\U, Z_{n_1, \ldots, n_p})\geq \M^s_N(\U, Z)-\sum_{i=1}^\infty\epsilon_i$ for any $p\in \N$.
Let
$$
K=\bigcap_{p=1}^\infty \overline{Z_{n_1,\ldots, n_p}}.$$
Since $\phi$ is continuous, we can show that $\bigcap_{p=1}^\infty \overline{Z_{n_1,\ldots, n_p}}=\bigcap_{p=1}^\infty Z_{n_1,\ldots, n_p}$ by applying Cantor's diagonal argument. Hence $K$ is a compact subset of $Z$. If $\Lambda \subset\bigcup_{j\geq
N}\W_j(\U)$ is a cover of $K$ (of course it is an open cover), then it is a cover of $\overline{Z_{n_1,\ldots, n_p}}$ when $p$ is large enough, which implies
$$\sum_{\bu\in \Lambda} e^{-sm(\bu)}\geq \lim_{p\to \infty}\M^s_N(\U, Z_{n_1, \ldots, n_p})\geq \M^s_N(\U, Z)-\sum_{i=1}^\infty\epsilon_i.$$
 Hence $\M^s_N(\U, K)\geq \M^s_N(\U, Z)-\sum_{i=1}^\infty\epsilon_i$. Since $\sum_{i=1}^\infty\epsilon_i$ can be chosen arbitrarily small, we prove (ii).

Now we turn to prove (i). Our argument is  modified from the classical proof of the ``increasing sets lemma'' for Hausdorff outer measures (cf. \cite[Sect. II]{Car67} and  \cite[Lemma 5.3]{Fal85}).
Note that any two non-empty elements in  $\W_n(\U)$ are disjoint,  and each element in $\W_{n+1}(\U)$ is a subset of some element in
$\W_n(\U)$. We call this the {\it net property } of $(\W_n(\U))$.

Let $E_i\uparrow E$ be given. Let $(\delta_i)$ be a sequence of positive numbers to be specified later and for each $i$,  choose a covering $\Lambda_i\subset \bigcup_{j\geq
N}\W_j(\U)$ of $E_i$ such that
\begin{equation}
\label{e-5.2}
\sum_{\bu\in \Lambda_i} e^{-sm(\bu)}\leq M^s_N(\U, E_i)+\delta_i.
\end{equation}
By the net property of
$(\W_n(\U))$, we may assume that for each $i$, the elements in $\Lambda_i$ are disjoint.

For any $x\in E$, choose $\bu_x\in \bigcup_{i=1}^\infty \Lambda_i$ containing $x$ such that $m(\bu_x)$ is the smallest. By the net property of
$(\W_n(\U))$, the collection $\{\bu_x:\; x\in E\}$ consists of  countable many disjoint elements.  Relabel these elements by $\bu_i$'s.  Clearly
$E\subset \bigcup_i\bu_i$.

We now choose an integer $k$.   Use $\A_1$ to denote the collection of those $\bu_i$'s  that are taken from $\Lambda_1$.  They cover a certain subset $Q_1$ of $E_k$.
The same subset is covered by a certain sub-collection of $\Lambda_k$, denoted as $\Lambda_{k,1}$. Since $\Lambda_{k,1}$ also covers the smaller set
$Q_1\cap E_1$, by \eqref{e-5.2},
\begin{equation}
\label{e-b1}
\sum_{\bu\in \A_1}e^{-sm(\bu)}\leq \sum_{\bu\in \Lambda_{k,1}}e^{-sm(\bu)}+\delta_1.
\end{equation}
To see this, assume that \eqref{e-b1} is false. Then by \eqref{e-5.2},
$$
\sum_{\bu\in (\Lambda_1\backslash \A_1)\cup \Lambda_{k,1}}e^{-sm(\bu)}< M^s_N(\U, E_1),
$$
which contradicts the fact that $(\Lambda_1\backslash \A_1)\cup \Lambda_{k,1}\subset \bigcup_{j\geq N}\W_j(\bu)$ is an open cover of $E_1$.
Next we use $\A_2$ to denote the collection of those $\bu_i$'s  that are taken from $\Lambda_2$ but not from $\Lambda_1$. Define $\Lambda_{k,2}$ similarly. As above, we find
\begin{equation}
\label{e-b2}
\sum_{\bu\in \A_2}e^{-sm(\bu)}\leq \sum_{\bu\in \Lambda_{k,2}}e^{-sm(\bu)}+\delta_2.
\end{equation}
We repeat the argument until all coverings $\Lambda_n$, $n\leq k$, have been considered. Note that
$\bigcup_{\bu\in\Lambda_{k,i}}\bu\subseteq \bigcup_{\bu\in \A_i}\bu$ for $i\leq k$. For different $i,i'\leq k$, the elements in $\Lambda_{k,i}$ are disjoint from those in  $\Lambda_{k,i'}$.
 The $k$ inequalities \eqref{e-b1}, \eqref{e-b2}, \ldots, are added which yields
$$\sum_{\bu\in \bigcup_{n=1}^k \A_n} e^{-sm(\bu)}\leq \sum_{\bu\in \bigcup_{n=1}^k\Lambda_{k,n}}e^{-sm(\bu)}+\sum_{n=1}^k\delta_n\leq \M^s_N(\U, E_k)+\sum_{n=1}^k\delta_n+\delta_k.$$
Letting $k\to \infty$, we have
$$\sum_{i}e^{-sm(\bu_i)}\leq \lim_{k\to \infty} \M^s_N(\U, E_k)+\sum_{n=1}^\infty\delta_n.$$
Since $\sum_{n=1}^\infty \delta_n$ can be chosen arbitrarily small we have
$$
\M^s_N(\U, E)\leq \lim_{k\to \infty} \M^s_N(\U, E_k).
$$
Since the opposite inequality is trivial we have proved (i).
\end{proof}

\begin{proof}[Proof of Theorem \ref{thm-6.1}] Let $Z$ be an analytic subset of $X$ with $\htop^B(T, Z)>0$. Let $0<s<\htop^B(T, Z)$. By \eqref{e-e1},
there exists a closed-open partition $\U$ so that $\htop^B(T,\U, Z)>s$
and thus $\M^s(\U,Z)=\infty$. Hence $\M_N^s(\U,Z)>0$ for some $N\in
\N$. By Proposition \ref{pro-5.2}, we can find a compact set
$K\subset Z$ such that  $\M_N^s(\U,K)>0$. It implies $\htop^B(T, K)\geq
\htop^B(T,\U, K)\geq  s$.
\end{proof}

Before we prove Theorem \ref{thm-1.1}(ii), we still need some notation and additional results.

Let us define the {\it natural extension}
$(\widetilde{X},\widetilde{T})$ of a TDS $(X,T)$ with a metric $d$
and a surjective map $T$ where $\widetilde{X}=\{ (x_1,x_2,\cdots):
T(x_{i+1})=x_i, x_i\in X, i\in \N \}$ is a subspace of the
product space $X^{\N}=\Pi_{i=1}^\infty X$ endowed with the
compatible metric $d_T$ as
$$d_T((x_1,x_2,\cdots),(y_1,y_2,\cdots))=\sum_{i=1}^\infty
\frac{d(x_i,y_i)}{2^i},$$ $\widetilde{T}:\widetilde{X}\rightarrow
\widetilde{X}$ is the shift homeomorphism with
$\widetilde{T}(x_1,x_2,\cdots)=(T(x_1),x_1,x_2,\cdots)$, and
$\pi_i:\widetilde{X}\rightarrow X$ is the projection to the $i$-th
coordinate. Clearly, $\pi_i:(\widetilde{X},\widetilde{T})\rightarrow
(X,T)$ is a factor map.
\begin{lem} \label{lem-nat} Let $(X,T)$ be a TDS with a metric
$d$ and a surjective map $T$, $(\widetilde{X},\widetilde{T})$ be the
natural extension of $(X,T)$ and $\pi_1:\widetilde{X}\rightarrow X$
be the projection to the first coordinate. Then $\sup_{x\in
X}\htop^{UC}(\widetilde{T},\pi_1^{-1}(x))=0$.
\end{lem}
\begin{proof} Fix $x\in
X$. For any $\epsilon>0$, take $N\in \N$ large enough such that
$\sum_{i=N}^{\infty} \frac{{\rm diam}(X)}{2^i}<\epsilon$.

Let $E_N\subseteq \pi_1^{-1}(x)$ be a finite $(N,\epsilon)$-spanning
set of $\pi_1^{-1}(x)$. Next we are to show that $E_N$ is also a
$(n,\epsilon)$-spanning set of $\pi_1^{-1}(x)$ for $n>N$.

Fix $n\in \N$ with $n>N$. For any $\widetilde{y}\in
\pi_1^{-1}(x)$, since $E_N$ is a $(N,\epsilon)$-spanning set of
$\pi_1^{-1}(x)$ there exist $\widetilde{x}\in E_N$ such that
$d_T(\widetilde{T}^i \widetilde{x}, \widetilde{T}^i
\widetilde{y})<\epsilon$ for $i=0,1,\cdots,N-1$. Now for $k\in \{
N,N+1,\cdots,n-1 \}$, we have $\pi_j(\widetilde{T}^k
\widetilde{x})=\pi_j(\widetilde{T}^k \widetilde{y}))=T^{k-j+1}(x)$
for $j=1,\cdots, k,k+1$. Thus
\begin{eqnarray*}d_T(\widetilde{T}^k
\widetilde{x}, \widetilde{T}^k \widetilde{y})&=&\sum_{j=1}^\infty
\frac{d(\pi_j(\widetilde{T}^k \widetilde{x}),
\pi_j(\widetilde{T}^k\widetilde{x}))}{2^j}=\sum_{j=k+2}^{\infty}
\frac{d(\pi_j(\widetilde{T}^k \widetilde{x}),
\pi_j(\widetilde{T}^k\widetilde{x}))}{2^j}\\
&\le & \sum_{j=k+2}^{\infty} \frac{\mbox{diam}(X)}{2^j}\le
\sum_{j=N}^{\infty} \frac{\mbox{diam}(X)}{2^j}<\epsilon.
\end{eqnarray*}
This implies $(d_T)_n(\widetilde{x}, \widetilde{y})<\epsilon$. Hence
$E_N$ is also a $(n,\epsilon)$-spanning set of $\pi_1^{-1}(x)$ for
$n>N$. Let $\tilde{r}_n(\pi^{-1}(x),\epsilon)$ denote the smallest cardinality of $(n,\epsilon)$-spanning sets of $\pi^{-1}(x)$. Then
$\tilde{r}_n(\pi^{-1}(x),\epsilon)\leq \#(E_N)$. Hence
$$\htop^{UC}(\widetilde{T},\pi_1^{-1}(x))=\lim_{\epsilon\to 0}\limsup\limits_{n\rightarrow \infty} \frac{1}{n} \log \tilde{r}_n(\pi^{-1}(x),\epsilon)\le
\lim_{\epsilon\to 0}\limsup\limits_{n\rightarrow \infty} \frac{1}{n} \log \#(E_N)=0.$$
 This ends the proof of the lemma.
\end{proof}

In the following part we will lift general  TDSs having finite
topological entropy to zero dimensional TDSs by the so called {\it principal
extensions}.
\begin{de} \label{de-7-1} \cite{L} An extension
$\pi: (Z,R)\rightarrow (X, T)$ between two TDSs is a principal
extension if $h_\nu(R)=h_{\nu\circ \pi^{-1}}(T)$ for every $\nu \in M(Z,R)$.
\end{de}

The following general result is needed in our proof of Theorem \ref{thm-1.1}(ii).
\begin{pro}[Proposition 7.8 in \cite{BD}] \label{BD1} Every invertible TDS $(X, T)$ with
$h_{top}(T)<\infty$ has a zero dimensional principal extension
$(Z,R)$ with $R$ being invertible.
\end{pro}

Let $\pi:(Y,S)\rightarrow (X,T)$ be a factor map between two TDSs.
Bowen proved that $\htop(S)\le \htop(T)+\sup_{x\in X}
\htop^{UC}(S,\pi^{-1}(x))$ (cf. \cite[Theorem 17]{Bow71}). In fact, Bowen's proof is also valid for the
following result (see, i.e.  Theorem 7.3 in \cite{HYZ-1} for a detailed proof).

\begin{thm} \label{thm-fe}Let $\pi: (X,T)\rightarrow (Y,S)$ be a
factor map between two TDSs. Then for any $E\subseteq X$ one has
\begin{equation}\label{eq-kkkk-1}
\htop^B(S,\pi(E))\le \htop^B(T,E)\le \htop^B(S,\pi(E))+\sup_{y\in
Y}\htop^{UC}(T,\pi^{-1}(y)).
\end{equation}
\end{thm}

We also need the following variational principle of conditional entropies.

\begin{pro} \label{pro-fe} Let $\pi: (X,T)\rightarrow (Y,S)$ be a
factor map between two TDSs. Then we have
\begin{equation}\label{eq-kkkk-1}
\sup_{y\in
Y}\htop^{UC}(T,\pi^{-1}(y))=\sup_{\mu\in M(X,T)} (h_\mu(T)-h_{\mu\circ \pi^{-1}}(S)).
\end{equation}
\end{pro}
\begin{proof}
It is  the direct combination of \cite[Theorem 3]{DoSe02} and \cite[Theorem 2.1]{LeWa77}.
\end{proof}

\begin{lem}\label{prop-zd-ex-e} Let $(X,T)$ be a TDS with $\htop(T)<\infty$.
Then there exists a factor map $\pi:(H,\Gamma)\rightarrow (X,T)$ such
that $(H,\Gamma)$ is zero dimensional and $$\sup_{x\in X}
\htop^{UC}(\Gamma,\pi^{-1}(x))=0.$$
\end{lem}
\begin{proof} First, we take $D=\{ \frac{1}{n}\}_{n\in \mathbb{N}}\cup \{ 0\}$ and let
$Z=X\times D$. Define $R:Z\rightarrow Z$ satisfying
$R(x,\frac{1}{n+1})=(x,\frac{1}{n})$, $n\in \mathbb{N}$;
$R(x,1)=(Tx,1)$ and $R(x,0)=(x,0)$ for $x\in X$. Then $(Z,R)$ is a
TDS and $R$ is surjective. If we identity $(x,1)$ with $x$ for each
$x\in X$, then $X$ can be viewed as a closed subset of $Z$ and
$R|_X=T$. It is also clear that
$\htop(R)=\htop(T)<\infty$.

Let $(\widetilde{Z},\widetilde{R})$ be the natural extension of
$(Z,R)$ and $\pi_1:\widetilde{Z}\rightarrow Z$ be the projection to
the first coordinate. Then
\begin{align}\label{u=l-eq2}
\sup_{z\in Z}\htop^{UC}(\widetilde{R},\pi_1^{-1}(z))=0
\end{align}
 by Lemma \ref{lem-nat}, and,
so $h_{\text{top}}(\widetilde{R})=h_{\text{top}}(R)<\infty$. Since
$\widetilde{R}$ is homeomorphism on $\widetilde{Z}$, by Lemma \ref{BD1}, there exists a
factor map $\psi:(W,G)\rightarrow (\widetilde{Z},\widetilde{R})$
such that $(W,G)$ is a zero-dimensional TDS and
$\psi$ is principal extension.

Since $h_{\text{top}}(\widetilde{R})<\infty$ and $\psi$ is principal
extension, we have the following variational principle of condition entropy
\begin{align}\label{u=l-eq1}
\sup_{\widetilde{z}\in \widetilde{Z}}
\htop^{UC}(G,\psi^{-1}(\widetilde{z}))=\sup_{\theta\in M(W,G)}
(h_\theta(G)-h_{ \theta\circ\psi^{-1}}(\widetilde{R}))=0.
\end{align}
The first equality in \eqref{u=l-eq1} follows from \eqref{eq-kkkk-1}.

Let $H=\psi^{-1}(\pi^{-1}_1X)$, $\Gamma=G|_H$ and $\pi=\pi_1\circ \psi
|_H$. Then $(H,\Gamma)$ be a zero-dimensional TDS and
$\pi:(H,\Gamma)\rightarrow (X,T)$ be a factor map. Applying Proposition \ref{pro-fe} to the factor map $\pi:(H,\Gamma)\rightarrow (X,T)$, we obtain
\begin{eqnarray*}
\sup_{x\in X} \htop^{UC}
(\Gamma,\pi^{-1}(x)) &=& \sup_{\mu\in M(H, \Gamma)}(h_\mu(\Gamma)-h_{\mu\circ \pi^{-1}}(T))\\
&\le & \sup_{\mu\in M(W, G)}(h_\mu(\Gamma)-h_{\mu\circ \pi^{-1}}(T))\\
&= & \sup_{\mu\in M(W, G)}(h_\mu(\Gamma)-h_{\mu\circ \psi^{-1}}(T)+ h_{\mu\circ \psi^{-1}}(T)-h_{\mu\circ \pi^{-1}}(T))\\
&\leq & \sup_{\mu\in M(W, G)}(h_\mu(\Gamma)-h_{\mu\circ \psi^{-1}}(T))+ \sup_{\nu\in (\widetilde{Z}, \widetilde{R})}(h_\nu(\widetilde{R})-h_{\nu\circ \pi_1^{-1}}(R))\\
&= & \sup_{\widetilde{z}\in \widetilde{Z}} \htop^{UC}(G,\psi^{-1}(\widetilde{z}))
+\sup_{z\in Z} \htop^{UC}(\widetilde{R},\pi_1^{-1}(z))\\
&=& 0 \qquad \qquad (\mbox{ by \eqref{u=l-eq1}, \eqref{u=l-eq2}}).
\end{eqnarray*}
This shows $\sup_{x\in X} \htop^{UC}(\Gamma,\pi^{-1}(x))=0$.
\end{proof}

\medskip

\begin{proof}[Proof of Theorem \ref{thm-1.1}(ii)] By Lemma \ref{prop-zd-ex-e}, there exists a factor map $\pi:(Y,S)\rightarrow (X,T)$ such
that $(Y,S)$ is zero dimensional and $\sup_{x\in X}
\htop^{UC}(S,\pi^{-1}(x))=0$. By Theorem \ref{thm-fe}, we have that for any
$F\subset Y$,
\begin{equation}\label{eq-lift-eq}
\htop^B(S,F)=\htop^B(T,\pi(F)).
\end{equation}

Let  $Z$ be an analytic subset of $X$. Then $\pi^{-1}(Z)$ is also
an analytic set of $Y$ (cf. Federer \cite[2.2.10]{Fed69}). By \eqref{eq-lift-eq} and Theorem \ref{thm-6.1},
\begin{align*}
\htop^B(T,Z)&=\htop^B(S,\pi^{-1}(Z))=\sup\{\htop^B(S,E):\; E\subseteq
\pi^{-1}(Z),\; E
\mbox{ is compact}\}\\
&=\sup\{\htop^B(T,\pi(E)):\; E\subseteq \pi^{-1}(Z),\; E \mbox{ is
compact}\} \\
&\le \sup\{\htop^B(T,K):\; K\subseteq Z,\; K \mbox{ is compact}\}.
\end{align*}
The reverse inequality is trivial, so $$\htop^B(T,Z)=
\sup\{\htop^B(T,K):\; K\subseteq Z,\; K \mbox{ is compact}\}.$$ This
finishes the proof.
\end{proof}

\begin{rem}\label{rem-1}
{\rm For an invertible TDS $(X, T)$,
Lindenstrauss and Weiss \cite{LiW} introduced the mean dimension
$mdim(X,T)$ (an idea suggested by Gromov). It is well known that for
an invertible TDS $(X, T)$, if $\htop(T)<\infty$ or the
topological dimension of $X$ is finite, then $mdim(X,T)=0$ (see
\cite[Definition 2.6 and Theorem 4.2]{LiW}).

In general, one can show that for an invertible TDS $(X, T)$, if
$mdim(X,T)=0$ then $(X,T)$ has a zero dimensional principal
extension $(Z,R)$ with $R$ being invertible. Indeed, let $(Y,S)$ be
an irrational rotation on the circle. Then $(X\times Y,T\times S)$
admits a nonperiodic minimal factor $(Y,S)$ and $mdim(X\times
Y,T\times S)=0$. Hence $(X\times Y,T\times S)$ has the so called
small boundary property \cite[Theorem 6.2]{Li}, which implies the
existence of a basis of the topology consisting of sets whose
boundaries have measure zero for every invariant measure. With these
results it is easy to construct a refining sequence of
small-boundary partitions for $(X\times Y,T\times S)$, where the
partitions have small boundaries if their boundaries have measure
zero for all $\mu\in \mathcal{M}(X\times Y, T\times S)$. Then by a
standard construction (see p. 152-153 in \cite{BD}), which
associates to this sequence a zero dimensional principal extension
$(Z,R)$ of $(X\times Y,T\times S)$ with $R$ being invertible.
Finally note that $(X\times Y, T\times S)$ is a principal extension
of $(X,T)$, we know that $(Z,R)$ is also a zero dimensional
principal extension of $(X,T)$ since the composition of two
principal extensions is still a principal extension.
}
\end{rem}

\begin{rem}
\label{rem-lind}
{\rm By Remark \ref{rem-1}, we may strengthen Theorem \ref{thm-1.1}(ii) as follows:
Let $(X,T)$ be a TDS with $mdim(X,T)=0$. Then for any analytic set
$Z\subseteq X$,
$$\htop^B(T,Z)=\sup\{\htop^B(T,K):\; K\subseteq Z,\; K \mbox{ is compact}\}.$$
}
\end{rem}

\section{Variational principle for the packing topological entropy}
\label{S-4}
In this section we prove Theorem \ref{thm-4.1}. We first give a lemma.
\begin{lem}
\label{lem-5.1} Let $Z\subset X$ and $s,\epsilon>0$. Assume
$P^s_\epsilon(Z)=\infty$. Then for any given finite interval
$(a,b)\subset \R$ with $a\ge 0$ and any $N\in \N$, there exists a
finite disjoint collection $\{\overline{B}_{n_i}(x_i,\epsilon)\}$
such that $x_i\in Z$, $n_i\geq N$ and $\sum_{i}e^{-n_is}\in (a,b)$.
 \end{lem}
\begin{proof}
Take $N_1>N$ large enough such that $e^{-N_1s}<b-a$. Since $P^s_\epsilon(Z)=\infty$, we have $P^s_{N_1,\epsilon}(Z)=\infty$. Thus there is a finite
disjoint collection $\{\overline{B}_{n_i}(x_i,\epsilon)\}$ such that $x_i\in Z$, $n_i\geq N_1$ and $\sum_{i}e^{-n_is}>b$. Since $e^{-n_is}<b-a$, by discarding elements in this collection one by one until  we can have  $\sum_{i}'e^{-n_is}\in (a,b)$.
\end{proof}

\begin{proof}[Proof of Theorem \ref{thm-4.1}] We divide the proof into two parts:

{\it Part 1. $\htop^P(T, Z)\geq \sup \{\overline{h}_\mu(T): \mu \in M(X), \; \mu(Z)=1\}$ for any Borel set $Z\subseteq X$.}

To see this, let  $\mu\in \M(X)$ with $\mu(Z)=1$ for some Borel set $Z\subseteq X$. We need to show that $\htop^P(T, Z)\geq \overline {h}_\mu(T)$. For this purpose we may assume $\overline {h}_\mu(T)>0$;
otherwise we have nothing to prove. Let $0<s<\overline {h}_\mu(T)$. Then there exist $\epsilon,\delta>0$, and a Borel set $A\subset Z$ with $\mu(A)>0$ such that
$$
\overline {h}_\mu(T,x,\epsilon)>s+\delta,\quad \forall\; x\in A,
$$
where
$
\overline {h}_\mu(T,x,\epsilon):=\limsup_{n\to \infty}-\frac{1}{n}\log \mu(B_n(x,\epsilon)).
$

Next we show that ${\mathcal P}^s_{\epsilon/5}(Z)=\infty$, which implies that  $\htop^P(T, Z)\geq \htop^P(T, Z,\epsilon/5)\geq s$. To achieve this, it suffices to show  that $P^s_{\epsilon/5}(E)=\infty$ for any Borel $E\subset A$ with $\mu(E)>0$. Fix such a set $E$. Define
$$
E_n=\{x\in E:\; \mu(B_n(x,\epsilon))<e^{-n(s+\delta)}\},\quad n\in \N.
$$
Since $E\subset A$, we have $\bigcup_{n=N}^\infty E_n=E$ for each $N\in \N$.  Fix $N\in \N$. Then
$\mu(\bigcup_{n=N}^\infty E_n)=\mu(E)$, and hence there exists $n\geq N$ such that
$$
\mu(E_n)\geq \frac{1}{n(n+1)}\mu(E).
$$
Fix such $n$ and consider the family $\{B_n(x,\epsilon/5):\; x\in
E_n\}$. By Lemma \ref{lem-2.1} (in which we use $d_n$ instead of
$d$), there exists a finite pairwise disjoint family $\{B_n(x_i,
\epsilon/5)\}$ with $x_i\in E_n$ such that
$$\bigcup_iB_n(x_i, \epsilon)\supset \bigcup_{x\in E_n}B_n(x, \epsilon/5)\supset E_n.$$
Hence
\begin{eqnarray*}
P^s_{N, \epsilon/5}(E)&\geq & P^s_{N, \epsilon/5}(E_n)\geq \sum_{i}e^{-ns}\geq e^{n\delta}\sum_{i}e^{-n(s+\delta)} \\
&\geq & e^{n\delta} \sum_{i} \mu(B_n(x_i,\epsilon))\geq e^{n\delta}\mu(E_n)\geq \frac{e^{n\delta}}{n(n+1)}\mu(E).
\end{eqnarray*}
Since $\frac{e^{n\delta}}{n(n+1)}\to \infty$ as $n\to \infty$, letting $N\to \infty$ we obtain that $P^s_{\epsilon/5}(E)=\infty$.

{\it Part 2. Let $Z\subseteq X$ be analytic with $\htop^P(T, Z)>0$. For any $0<s<\htop^P(T, Z)$, there exists a compact set $K\subseteq Z$ and
$\mu\in M(K)$ such that $\overline {h}_\mu(T)\geq s$.}

Since $Z$ is  analytic, there exists a continuous surjective map $\phi:\;{\mathcal N}\to Z$.    Let $\Gamma_{n_1,n_2,\ldots,n_p}$ be the set of $(m_1,m_2,\ldots)\in {\mathcal N}$
such that $m_1\leq n_1$, $m_2\leq n_2$, $\ldots$, $m_p\leq n_p$ and let $Z_{n_1,\ldots, n_p}$ be the image of $\Gamma_{n_1,\ldots,n_p}$ under $\phi$.

Take $\epsilon>0$ small enough so  that  $0<s<\htop^P(T, Z,\epsilon)$.
Take $t\in (s, \htop^P(T, Z,\epsilon))$. We are going to construct inductively a sequence of finite sets $(K_i)_{i=1}^\infty$ and a sequence of finite measures $(\mu_i)_{i=1}^\infty$ so that $K_i\subset Z$ and $\mu_i$ is supported on $K_i$ for each $i$. Together with these two sequences, we construct also a sequence
of integers $(n_i)$, a sequence of positive numbers $(\gamma_i)$ and a sequence of integer-valued functions $(m_i:\; K_i\to \N)$.   The method of our construction is inspired somehow by  the work of Joyce and Preiss \cite{JoPr95} on packing measures.

The construction is divided into several small steps:

{\sl Step 1. Construct $K_1$ and $\mu_1$, as well as $m_1(\cdot)$, $n_1$ and $\gamma_1$.}

 Note that ${\mathcal P}^t_\epsilon(Z)=\infty$.
Let $$
H=\bigcup\{G\subset X:\; \mbox{ $G$ is open}, \;  {\mathcal P}^t_\epsilon(Z\cap G)=0\}.
$$
Then ${\mathcal P}^t_\epsilon(Z\cap H)=0$ by the separability of $X$. Let $Z'=Z\backslash H=Z\cap (X\backslash H)$.
For any open set $G\subset X$, either $Z'\cap G=\emptyset$, or ${\mathcal P}^t_\epsilon(Z'\cap G)>0$. To see this, assume ${\mathcal P}^t_\epsilon(Z'\cap G)=0$ for an open set $G$; then
${\mathcal P}^t_\epsilon(Z\cap G)\leq {\mathcal P}^t_\epsilon(G\cap Z')+ {\mathcal P}^t_\epsilon(Z\cap H)=0$, implying $G\subset H$ and hence $Z'\cap G=\emptyset$.

  Note that  ${\mathcal P}^t_\epsilon(Z')={\mathcal P}^t_\epsilon(Z)=\infty$ (because  ${\mathcal P}^t_\epsilon(Z)\leq {\mathcal P}^t_\epsilon(Z')+
  {\mathcal P}^t_\epsilon(Z\cap H)={\mathcal P}^t_\epsilon(Z')$). It follows ${\mathcal P}^s_\epsilon(Z')=\infty$. By Lemma \ref{lem-5.1}, we can find a finite set $K_1\subset  Z'$, an integer-valued function $m_1(x)$ on $K_1$ such that
the collection $\{\overline{B}_{m_1(x)}(x,\epsilon)\}_{x\in K_1}$ is disjoint and
$$\sum_{x\in K_1} e^{-m_1(x)s}\in (1,2).$$
Define $\mu_1=\sum_{x\in K_1} e^{-m_1(x)s}\delta_x$, where $\delta_x$ denotes the Dirac measure at $x$.
 Take a small $\gamma_1>0$ such that for any function $z: \; K_1\to X$ with $d(x,z(x))\leq \gamma_1$, we have for each $x\in K_1$,
\begin{equation}
\label{e-tkey}
\Big(\overline{B}(z(x),\gamma_1)\cup\overline{B}_{m_1(x)}(z(x),
\epsilon)\Big)\cap \Big(\bigcup_{y\in K_1\backslash
\{x\}}\overline{B}(z(y),\gamma_1)\cup \overline{B}_{m_1(y)}(z(y),
\epsilon)\Big)=\emptyset.
\end{equation}
Here and afterwards, $\overline{B}(x,\epsilon)$ denotes the closed ball $\{y\in X:\; d(x,y)\leq \epsilon)\}$.
Since $K_1\subset Z'$, ${\mathcal P}^t_\epsilon(Z\cap B(x,\gamma_1/4))\geq {\mathcal P}^t_\epsilon(Z'\cap B(x,\gamma_1/4))>0$ for each $x\in K_1$.
Therefore  we can pick a large $n_1\in \N$ so that  $Z_{n_1}\supset K_1$ and ${\mathcal P}^t_\epsilon(Z_{n_1}\cap B(x,\gamma_1/4))>0$ for each $x\in K_1$.

{\sl Step 2. Construct $K_2$ and $\mu_2$, as well as $m_2(\cdot)$, $n_2$ and $\gamma_2$.}

By \eqref{e-tkey}, the family of  balls $\{\overline{B}(x,\gamma_1)\}_{x\in
K_1}$, are pairwise disjoint. For each $x\in K_1$, since ${\mathcal
P}^t_{\epsilon}(Z_{n_1}\cap B(x,\gamma_1/4))>0$,  we can construct
as Step 1,  a finite set $$E_2(x)\subset Z_{n_1}\cap
B(x,\gamma_1/4)$$ and an integer-valued function
$$m_2:\; E_2(x)\to \N\cap [\max \{m_1(y):\; y\in K_1\},\infty)$$    such that \begin{itemize}
\item[(2-a)] ${\mathcal P}^t_\epsilon(Z_{n_1}\cap G)>0$ for each open set $G$ with $G\cap E_2(x)\neq \emptyset$;
\item[(2-b)] The elements in $\{\overline{B}_{m_2(y)}(y,\epsilon)\}_{y\in E_2(x)}$ are disjoint,  and
$$
\mu_1(\{x\})< \sum_{y\in E_2(x)}e^{-m_2(y)s}<(1+2^{-2})\mu_1(\{x\}).
$$
\end{itemize}
To see it, we fix $x\in K_1$. Denote $F=Z_{n_1}\cap B(x,\gamma_1/4)$. Let $$H_x:=\bigcup\{G\subset X:\; G  \mbox{ is open }\; {\mathcal
P}^t_{\epsilon}(F\cap G)=0\}.$$  Set $F'=F\backslash H_x$.  Then as in Step 1, we can show that ${\mathcal
P}^t_{\epsilon}(F')={\mathcal P}^t_{\epsilon}(F)>0$ and furthermore, ${\mathcal
P}^t_{\epsilon}(F'\cap G)>0$ for any open set $G$ with $G\cap F'\neq \emptyset$. Note that ${\mathcal
P}^s_{\epsilon}(F')=\infty$ (since $s<t$), by Lemma  \ref{lem-5.1}, we can find a finite set $E_2(x)\subset F'$ and a map $m_2:\; E_2(x)\to
\N\cap [\max \{m_1(y):\; y\in K_1\},\infty)$ so that  (2-b) holds. Observe that if a open set $G$ satisfies  $G\cap E_2(x)\neq \emptyset$, then
 $G\cap F'\neq \emptyset$, and hence  ${\mathcal
P}^t_{\epsilon}(Z_{n_1}\cap G)\geq {\mathcal
P}^t_{\epsilon}(F'\cap G)>0$. Thus (2-a) holds.

Since the family $\{\overline{B}(x,\gamma_1)\}_{x\in
K_1}$ is disjoint, $E_2(x)\cap E_2(x')=\emptyset$ for different $x,x'\in K_1$. Define $K_2=\bigcup_{x\in K_1}E_2(x)$ and
$$
\mu_2=\sum_{y\in K_2}e^{-m_2(y)s}\delta_y.
$$
By \eqref{e-tkey} and (2-b), the elements in  $\{\overline{B}_{m_2(y)}(y,\epsilon)\}_{y\in
K_2}$ are pairwise disjoint. Hence we can take
$0<\gamma_2<\gamma_1/4$ such that for any function $z: \; K_2\to X$
with $d(x,z(x))<\gamma_2$ for $x\in K_2$, we have
\begin{equation}
\label{e-tkey1}
\Big(\overline{B}(z(x),\gamma_2)\cup\overline{B}_{m_2(x)}(z(x),
\epsilon)\Big)\cap \Big(\bigcup_{y\in K_2\backslash
\{x\}}\overline{B}(z(y),\gamma_2)\cup \overline{B}_{m_2(y)}(z(y),
\epsilon)\Big)=\emptyset
\end{equation}
for each $x\in K_2$.
Choose a large $n_2\in \N$ such that
$Z_{n_1,n_2}\supset K_2$ and ${\mathcal P}^t_\epsilon(Z_{n_1, n_2}\cap B(x,\gamma_2/4))>0$ for each $x\in K_2$.

{\sl Step 3. Assume that $K_i$, $\mu_i$, $m_i(\cdot)$, $n_i$ and $\gamma_i$ have been constructed for $i=1,\ldots, p$.
 In particular, assume that
  for any function $z: \; K_p\to X$ with $d(x,z(x))<\gamma_p$ for $x\in K_p$, we have
\begin{equation}
\label{e-tkeyp}
\Big(\overline{B}(z(x),\gamma_p)\cup\overline{B}_{m_p(x)}(z(x),
\epsilon)\Big)\cap \Big(\bigcup_{y\in K_p\backslash
\{x\}}\overline{B}(z(y),\gamma_p)\cup \overline{B}_{m_p(y)}(z(y),
\epsilon)\Big)=\emptyset
\end{equation}
for each $x\in K_p$; and $Z_{n_1,\ldots, n_p}\supset K_p$ and ${\mathcal P}^t_\epsilon(Z_{n_1,\ldots,  n_p}\cap B(x,\gamma_p/4))>0$ for each $x\in K_p$.
 We construct below each term of them for $i=p+1$ in a way similar to Step 2. }

Note that the elements in $\{\overline{B}(x,\gamma_p)\}_{x\in
K_p}$ are pairwise disjoint. For each $x\in K_p$, since ${\mathcal
P}^t_{\epsilon}(Z_{n_1,\ldots, n_p}\cap B(x,\gamma_p/4))>0$,  we can
construct as Step 2,  a finite set $$E_{p+1}(x)\subset
Z_{n_1,\ldots, n_p}\cap B(x,\gamma_p/4)$$ and an integer-valued
function
$$m_{p+1}:\; E_{p+1}(x)\to \N\cap [\max \{m_p(y):\; y\in K_p\},\infty)$$    such that
\begin{itemize}
\item[(3-a)]
 ${\mathcal P}^t_\epsilon(Z_{n_1,\ldots, n_p}\cap G)>0$ for each open set $G$ with $G\cap E_{p+1}(x)\neq \emptyset$; and\\
\item[(3-b)]
 $\{\overline{B}_{m_{p+1}(y)}(y,\epsilon)\}_{y\in E_{p+1}(x)}$ are disjoint and satisfy
$$
\mu_p(\{x\})< \sum_{y\in E_{p+1}(x)}e^{-m_{p+1}(y)s}<(1+2^{-p-1})\mu_p(\{x\}).
$$
\end{itemize}
Clearly $E_{p+1}(x)\cap E_{p+1}(x')=\emptyset$ for different $x,x'\in K_p$. Define $K_{p+1}=\bigcup_{x\in K_p}E_{p+1}(x)$ and
$$
\mu_{p+1}=\sum_{y\in K_{p+1}}e^{-m_{p+1}(y)s}\delta_y.
$$
By \eqref{e-tkeyp} and (3-b), $\{\overline{B}_{m_{p+1}(y)}(y,\epsilon)\}_{y\in K_{p+1}}$ are disjoint. Hence we can take $0<\gamma_{p+1}<\gamma_p/4$ such that
 for any function $z: \; K_{p+1}\to X$ with $d(x,z(x))<\gamma_{p+1}$, we have for each $x\in K_{p+1}$,
\begin{equation}
\label{e-tkey2}
\Big(\overline{B}(z(x),\gamma_{p+1})\cup\overline{B}_{m_{p+1}(x)}(z(x),
\epsilon)\Big)\cap \Big(\bigcup_{y\in K_{p+1}\backslash
\{x\}}\overline{B}(z(y),\gamma_{p+1})\cup
\overline{B}_{m_{p+1}(y)}(z(y), \epsilon)\Big)=\emptyset.
\end{equation}
Choose a large $n_{p+1}\in \N$ such that
$Z_{n_1,\ldots, n_{p+1}}\supset K_{p+1}$ and $${\mathcal P}^t_\epsilon(Z_{n_1,\ldots,  n_{p+1}}\cap B(x,\gamma_{p+1}/4))>0$$ for each $x\in K_{p+1}$.

As in the above steps, we can  construct by induction the sequences   $(K_i)$, $(\mu_i)$, $(m_i(\cdot))$, $(n_i)$ and $(\gamma_i)$.
We summarize some of their basic properties as follows:
\begin{itemize}
\item[(a)] For each $i$, the family ${\mathcal F}_i:=\{\overline{B}(x,\gamma_i):\; x\in K_i\}$ is disjoint.  Each element in ${\mathcal F}_{i+1}$ is a subset of $\overline {B}(x, \gamma_i/2)$ for some $x\in K_i$.
\item[(b)] For each $x\in K_i$ and $z\in \overline{B}(x,\gamma_i)$,
$$
\overline{B}_{m_i(x)}(z,\epsilon)\cap \bigcup_{y\in K_i\backslash \{x\}}\overline{B}(y,\gamma_i)=\emptyset \mbox{ and }
$$
\begin{equation*}
\label{e-ite}
\begin{split}
\mu_i(\overline{B}(x,\gamma_i))=e^{-m_i(x)s}\leq \sum_{y\in E_{i+1}(x)}e^{-m_{i+1}(y)s}\leq
(1+2^{-i-1})\mu_i(\overline{B}(x,\gamma_i)),
\end{split}
\end{equation*}
where $E_{i+1}(x)=B(x,\gamma_i)\cap K_{i+1}$.
\end{itemize}
The second part in (b) implies,
$$
\mu_{i}(F_i)\leq \mu_{i+1}(F_i)= \sum_{F\in {\mathcal F}_{i+1}:\; F\subset F_i}\mu_{i+1}(F)\leq (1+2^{-i-1})\mu_i(F_i),\qquad F_i\in {\mathcal F}_i
$$
Using the above inequalities repeatedly, we have for any $j>i$,
  \begin{equation}
 \label{e-ite1}\mu_i(F_i)\leq \mu_j(F_i)\leq \prod_{n=i+1}^j (1+2^{-n}) \mu_i(F_i)\leq C\mu_i(F_i),\quad \forall F_i\in {\mathcal F}_i,
 \end{equation}
 where $C:=\prod_{n=1}^\infty(1+2^{-n})<\infty$.

 Let $\tilde{\mu}$ be a limit point of $(\mu_i)$ in the weak-star topology. Let
 $$K=\bigcap_{n=1}^\infty \overline{\bigcup_{i\geq n} K_i}.$$
 Then
 $\mu$ is supported on $K$. Furthermore
 $$
 K=\bigcap_{n=1}^\infty \overline{\bigcup_{i\geq n} K_i}\subset \bigcap_{p=1}^\infty \overline{ Z_{n_1,\ldots, n_p}}.
 $$
 However by the continuity of $\phi$, we can show that   $\bigcap_{p=1}^\infty  Z_{n_1,\ldots, n_p}=\bigcap_{p=1}^\infty \overline{ Z_{n_1,\ldots, n_p}}$ by applying Cantor's diagonal argument. Hence $K$ is a compact subset of $Z$.

  On the other hand, by \eqref{e-ite1}, $$e^{-m_i(x)s}=\mu_i(\overline{B}(x,\gamma_i))\leq \tilde{\mu} ({B(x,\gamma_i)})\leq C \mu_i(\overline{B}(x,\gamma_i))=Ce^{-m_i(x)s},\quad
 \forall x\in K_i.$$
 In particular, $1\leq  \sum_{x\in K_1}\mu_1({B(x,\gamma_1)})\leq \tilde{\mu}(K)\leq \sum_{x\in K_1}C\mu_1({B(x,\gamma_1)})\leq 2C$.
   Note that $K\subset \bigcup_{x\in K_i}\overline{B}(x,\gamma_i/2)$. By the first part of (b),  for each $x\in K_i$ and $z\in \overline{B}(x,\gamma_i)$,
 $$\tilde{\mu}(\overline{B}_{m_i(x)}(z,\epsilon))\leq \tilde{\mu}(\overline{B}(x,\gamma_i/2))\leq Ce^{-m_i(x)s}.$$
 For each $z\in K$ and $i\in N$,  $z\in \overline{B}(x,\gamma_i/2)$ for some $x\in K_i$. Hence
 $$\tilde{\mu}({B_{m_i(x)}(z,\epsilon)})\leq  Ce^{-m_i(x)s}.$$

Define $\mu=\tilde{\mu}/\tilde{\mu}(K)$. Then $\mu\in M(K)$, and  for each $z\in K$, there exists a sequence $k_i\uparrow \infty$ such that
   $\mu({B_{k_i}(z,\epsilon)})\leq  Ce^{-k_is}/\tilde{\mu}(K)$.
 It follows that $\overline{h}_\mu(T)\geq s$.
\end{proof}

\section{Main notation and conventions}
\label{B}
For the reader's convenience, we summarize in Table \ref{table-1} the main notation and typographical conventions used in this paper.
\begin{table}
\centering
\caption{Main notation and conventions}
\vspace{0.05 in}
\begin{footnotesize}
\begin{raggedright}
\begin{tabular}{p{1.8 in} p{4 in} }
\hline \rule{0pt}{3ex}
%& Effect of increasing \\
%Definition & substrate stiffness \\[3pt]
%\hline \rule{0pt}{3ex}
$(X,T)$ & A topological dynamical system (Sect. \ref{S-1})\\
$M(X)$ & Set of all  Borel probability measures on $X$ \\
$M(X,T),\; E(X,T)$ &  Set of $T$-invariant (resp. ergodic)  Borel probability measures on $X$\\
$d_n$ & $n$-th Bowen's metric (cf. \eqref{e-2.1})  \\
$B(x, \epsilon), \;\overline{B}(x,\epsilon)$ & Open (resp. closed)  ball in $(X, d)$ centered at $x$ of radius $\epsilon$\\
$B_n(X,\epsilon),\; \overline{B}_n(x,\epsilon)$ & Open (resp. closed)  ball in $(X, d_n)$ centered at $x$ of radius $\epsilon$\\
$\overline{h}_\mu(T)$, \;$\underline{h}_\mu(T)$ & Measure-theoretic upper (resp.  lower)  entropy of $T$ with respect to $\mu\in M(X)$ (Sect. \ref{S-1})\\
$\htop^{UC}(T, Z)$ & Upper capacity topological entropy of $Z$ (Sect. \ref{S-2})\\
$\htop^{B}(T, Z)$ & Bowen's  topological entropy of  of $Z$ (Sect. \ref{S-2})\\
$\htop^{P}(T, Z)$ & Packing topological entropy of  $Z$ (Sect. \ref{S-2})\\
$\htop(T)$ & Topological entropy of  $T$ (Sect. \ref{S-2})\\
${\M}_{N,\epsilon}^s(Z),\; {\M }_{\epsilon}^s(Z),\; {\M }^s(Z)$  &  (Sect.~\ref{S-2})\\
${\W }_{N,\epsilon}^s(Z), \;{\W }_\epsilon^s(Z), \;{\W }^s(Z)$ &  (Sect.~\ref{S-2})\\\
${ P}_{N,\epsilon}^s(Z)$,\; ${P}_\epsilon^s(Z)$,\; ${\mathcal P}_\epsilon^s(Z)$ &  (Sect.~\ref{S-2})\\
${\M}_{N}^s({\mathcal U}, Z)$,\; ${\M}^s({\mathcal U}, Z)$ &  (Sect.~\ref{S-2})\\
$\htop^B(T, {\mathcal U}, Z)$ &  (Sect.~\ref{S-2})\\
${\mathcal N}$ & the set of infinite sequences of
natural numbers endowed with  product topology.\\

\hline
\end{tabular}
\label{table-1}
\end{raggedright}
\end{footnotesize}
\end{table}

\noindent {\bf Acknowledgements}  The first author was partially
supported by the RGC grant and the Focused Investments Scheme B  in
CUHK.  The second author was partially supported by NSFC
(10911120388, 11071231), Fok Ying Tung Education Foundation,
FANEDD (Grant 200520), and the Fundamental Research Funds for the
Central Universities (WK0010000001,WK0010000014). The authors thank
Hanfeng Li for helpful comments.

\end{document}